\documentclass[reqno,12pt]{amsart}
\usepackage{amsmath,amsthm,amssymb,amsfonts,amscd,bm}
\usepackage{epsfig}
\usepackage{xypic}
\usepackage{color}
\numberwithin{equation}{section}

\theoremstyle{plain}
\newtheorem{theorem}{Theorem}

\newtheorem{corollary}[theorem]{Corollary}

\newtheorem*{theorem*}{Theorem}
\newtheorem*{conjecture*}{Conjecture}

\theoremstyle{definition}
\newtheorem{remark}[theorem]{Remark}
\newtheorem{example}[theorem]{Example}
\newtheorem*{definition}{Definition}

\newcommand{\CC}{{\mathbb{C}}}

\newcommand{\QQ}{{\mathbb{Q}}}

\newcommand{\ZZ}{{\mathbb{Z}}}

\newcommand{\eps}{\varepsilon}

\def \mf#1#2#3#4{
\xymatrix{
{#1}\  \ar@<0.4ex>[r]^{{#2}} & \ {#4}
\ar@<0.4ex>[l]^{{#3}}
}
}

\begin{document}
\title[Matrix factorizations]{Graded matrix factorizations of size two and reduction}
\author{Wolfgang Ebeling and Atsushi Takahashi}
\address{Institut f\"ur Algebraische Geometrie, Leibniz Universit\"at Hannover, Postfach 6009, D-30060 Hannover, Germany}
\email{ebeling@math.uni-hannover.de}
\address{
Department of Mathematics, Graduate School of Science, Osaka University, 
Toyonaka Osaka, 560-0043, Japan}
\email{takahashi@math.sci.osaka-u.ac.jp}
\subjclass[2010]{32S05, 13C14, 18G80}
\date{}
\begin{abstract} 
We associate a complete intersection singularity to a graded matrix factorization of size two of a polynomial in three variables. We show that we get an inverse to the reduction of singularities considered by C.~T.~C.~Wall. We study this for the full strongly exceptional collections in the triangulated category of graded matrix factorizations constructed by H.~Kajiura, K.~Saito, and the second author.
\end{abstract}
\maketitle

\section{Introduction}
In his classification of singularities \cite{WPLMS}, C.~T.~C.~Wall noticed that there is a relation between series of singularities given by a reduction. More precisely, he considers a linear reduction $L$ of a complete intersection singularity in $\CC^4$ to a hypersurface singularity in $\CC^3$ and a reduction $\Delta$ of a hypersurface singularity in $\CC^3$ of corank 3 to one of corank 2.  The aim of this paper is to show that as an inverse of such a reduction one can consider a graded matrix factorization of size two.

We consider weighted homogeneous deformations of an invertible polynomial. We use the approach of \cite{AT} to define graded matrix factorizations for such deformations. We classify weighted homogeneous deformations of invertible polynomials of three variables and we define Dolgachev numbers for them. We consider a graded matrix factorization of size two of such a deformation and associate a complete intersection singularity to it. We show that we get the original singularity back by $L$- or $\Delta$-reduction. Moreover we consider graded matrix factorizations of size two of invertible polynomials in two variables and show that they are related to a reduction between hypersurface singularities in $\CC^2$ considered in \cite{EW}.

In \cite{KST2}, H.~Kajiura, K.~Saito, and the second author constructed a full strongly exceptional collection in the triangulated category of graded matrix factorizations of a polynomial associated to a regular system of weights  whose smallest exponents are equal to $-1$. These polynomials define the 14 exceptional unimodal hypersurface singularities, the 6 heads of the bimodal series of hypersurface singularities, and 2 further singularities. They are  weighted homogeneous deformations of invertible polynomials of three variables. The exceptional collection is described by a certain quiver. The graded matrix factorizations of size two listed in the paper correspond to the outermost vertices of certain arms of the quiver. We show that they correspond to singularities where the corresponding length of the arm (which is one of the Dolgachev numbers) is increased by one. We thus obtain the picture of Fig.~\ref{FigPyramidR}.

There is Arnold's strange duality between the 14 exceptional unimodal hypersurface singularities and the extension \cite{EW} of it to the bimodal series and the complete intersection singularities which reduce to the 14 exceptional unimodal singularities and the 6 heads of the bimodal series. This can be considered as a special kind of mirror symmetry. As a corollary of our main result we get that the reduction is mirror dual to a special adjacency between these singularities.

\begin{sloppypar}

{\bf Acknowledgements}.\  
This work has been partially supported by DFG.
The second named author is also supported 
by JSPS KAKENHI Grant Numbers JP16H06337 and JP21H04994. 
\end{sloppypar}

\section{Graded matrix factorizations for invertible polynomials}
A polynomial $f(x_1, \ldots, x_n) \in \CC[x_1, \ldots , x_n]$ is called {\em weighted homogeneous}, if there are positive integers $w_1, \ldots, w_n$ and $d$ such that $f(\lambda^{w_1}x_1, \ldots , \lambda^{w_n} x_n)=\lambda^d f(x_1, \ldots , x_n)$ for $\lambda \in \CC^*$. We call $(w_1, \ldots, w_n;d)$ a {\em system of weights}. If ${\rm gcd}(w_1, \ldots, w_n,d)=1$, then the system of weights is called {\em reduced}.

\begin{definition}
A weighted homogeneous polynomial $f(x_1, \ldots, x_n)$ is called {\em invertible} if the following conditions are satisfied
\begin{itemize}
\item[(i)] it can be written
\[ f(x_1, \ldots , x_n)=\sum_{i=1}^n a_i \prod_{j=1}^n x_j^{E_{ij}},
\]
where $a_i \in \CC^\ast$, $E_{ij}$ are non-negative integers, and the $n \times n$-matrix $E:=(E_{ij})$ is invertible over $\QQ$,
\item[(ii)] it has an isolated singularity at the origin.
\end{itemize}
\end{definition}

Let $f(x_1,\dots ,x_n)=\sum_{i=1}^na_i\prod_{j=1}^nx_j^{E_{ij}}$ be an invertible polynomial. Without loss of generality, we assume that $\det E>0$. The {\em canonical system of weights} is the system of weights $(w_1, \ldots, w_n;d)$  given by the unique solution of the equation
\begin{equation*}
E
\begin{pmatrix}
w_1\\
\vdots\\
w_n
\end{pmatrix}
={\rm det}(E)
\begin{pmatrix}
1\\
\vdots\\
1
\end{pmatrix}
,\quad 
d:={\rm det}(E).
\end{equation*}
The canonical system of weights $(w_1, \ldots, w_n;d)$ is in general non-reduced, define
\[
c_f := {\rm gcd}(w_1, \ldots, w_n,d).
\]

\begin{definition}
A {\em weighted homogeneous deformation of an invertible polynomial} is a polynomial of the form
\[ f(x_1, \ldots , x_n)= \sum_{i=1}^N a_i \prod_{j=1}^n x_j^{E_{ij}}, 
\]
where $N \geq n$, $a_i \in \CC^\ast$, $E_{ij}$ are non-negative integers, $\sum_{i=1}^n a_i \prod_{j=1}^n x_j^{E_{ij}}$ is an invertible polynomial with canonical weight system $(w_1, \ldots , w_n;d)$, and $\prod_{j=1}^n x_j^{E_{ij}}$ are monomials of the same degree $d$ for $1 \leq i \leq N$.
\end{definition}

Note that the case $N=n$ is permitted, i.e., the weighted homogeneous deformations of an invertible polynomial include the invertible polynomial itself.

\begin{definition}
Let $f(x_1,\dots ,x_n)=\sum_{i=1}^N a_i\prod_{j=1}^nx_j^{E_{ij}}$ be a weighted homogeneous deformation of an invertible polynomial. Consider the free abelian group $\oplus_{i=1}^n\ZZ\vec{x_i}\oplus \ZZ\vec{f}$ 
generated by the symbols $\vec{x_i}$ for the variables $x_i$ for $i=1,\dots, n$
and the symbol $\vec{f}$ for the polynomial $f$.
The {\em maximal grading} $L_f$ of the invertible polynomial $f$ 
is the abelian group defined by the quotient 
\[
L_f:=\bigoplus_{i=1}^n\ZZ\vec{x_i}\oplus \ZZ\vec{f}\left/I_f\right.,
\]
where $I_f$ is the subgroup generated by the elements 
\[
\vec{f}-\sum_{j=1}^nE_{ij}\vec{x_j},\quad i=1,\dots ,N.
\]
\end{definition}

Note that $L_f$ is an abelian group of rank 1 which is not necessarily free. We have the degree map ${\rm deg}: L_f \to \ZZ$ defined by $\vec{x}_i \mapsto w_i/c_f$, $\vec{f} \mapsto d/c_f$, where $(w_1, \ldots , w_n; d)$ is the canonical system of weights of $f$. This map is an isomorphism if and only if the canonical system of weights is reduced, see \cite[Proposition~17]{ET}.
 
Let $f(x_1, \ldots, x_n)$ be a weighted homogeneous deformation of an invertible polynomial.
We recall the definition of an $L_f$-graded matrix factorization of $f$, see \cite{AT} and the references therein.

Let $S=\CC[x_1, \ldots , x_n]$. This is naturally an $L_f$-graded algebra:
\[
S= \bigoplus_{\vec{l} \in L_f} S_{\vec{l}}.
\]
For any $L_f$-graded $S$-module $M$, let $M(\vec{l})$ denote the $L_f$-graded $S$-module, where the grading is shifted by $\vec{l} \in L_f$:
\[ 
M(\vec{l})= \bigoplus_{\vec{l}' \in L_f}  M(\vec{l})_{\vec{l}'}, \quad M(\vec{l})_{\vec{l}'}:=M_{\vec{l}+\vec{l}'}.
\]
The grading shift by $\vec{l} \in L_f$ induces a functor $(\vec{l})$ on the category of finitely generated $L_f$-graded $S$-modules.

\begin{definition} An {\em $L_f$-graded matrix factorization} of $f$ is
\[
\overline{F}:=\Big(\mf{F_0}{f_0}{f_1}{F_1}\Big),
\]
where $F_0$ and $F_1$ are $L_f$-graded free $S$-modules of finite rank and $f_0: F_0 \to F_1$, $f_1: F_1 \to F_0(\vec{f})$ are morphisms such that $f_1 \circ f_0= f \cdot {\rm id}_{F_0}$ and $f_0(\vec{f}) \circ f_1= f \cdot {\rm id}_{F_1}$. The rank of $F_1$ is equal to the rank of $F_0$ and it is called the {\em size} or the {\em rank} of the matrix factorization $\overline{F}$. 
\end{definition}

Let $f(x_1, \ldots, x_n)$  be a weighted homogeneous deformation of an invertible polynomial with reduced weight system $(w_1, \ldots , w_n;d)$.  We define 
\[ 
\eps_f:= w_1 + \cdots + w_n -d.
\]

Let $f(x,y,z)$ be a weighted homogeneous deformation of an invertible polynomial in three variables.
We shall associate Dolgachev numbers to $f$. This is done as follows. Let $R_f:=\CC[x,y,z]/(f)$. Let $r \geq 3$, $A=(\alpha_1, \ldots , \alpha_r)$ and let $\bm{\lambda}=(\lambda_3, \ldots , \lambda_r)$ be an $(r-2)$-tuple of complex numbers with $\lambda_3=1$, $\lambda_i\neq \lambda_j$ for $i \neq j$, $i,j=3, \ldots, r$. Let $R_{A,\bm{\lambda}}$ be the ring associated to the weighted projective line corresponding to the tuple of numbers $A=(\alpha_1, \ldots , \alpha_r)$ \cite{GL}, namely the factor algebra 
\[R_{A,\bm{\lambda}}:=\CC[X_1, \ldots , X_r]/I_{\bm{\lambda}}
\]
where $I_{\bm{\lambda}}$ is the ideal generated by the polynomials
\[ 
X_i^{\alpha_i} - X_2^{\alpha_2} + \lambda_i X_1^{\alpha_1},  \quad i=3, \ldots , r.
\]

We shall consider a classification of weighted homogeneous deformations of invertible polynomials according to the classification of invertible polynomials in three variables in \cite{ET}. In each type except V, two of the three monomials are replaced by a product of $m$ terms involving parameters $(\lambda_3, \ldots , \lambda_{m+2})$. In the case of type V, we set $m=1$.

\begin{definition} The {\em Dolgachev numbers} of a weighted homogeneous deformation $f$ of an invertible polynomial is the $(m+2)$-tuple ($m \geq 1$) of numbers $A=(\alpha_1,\alpha_2, \underbrace{\alpha_3, \ldots, \alpha_3}_m)$  defined by the embedding $R_f \hookrightarrow R_{A,\bm{\lambda}}$ given by Table~\ref{TabGenerator}.
\end{definition}
\begin{table}[h]
\begin{center}
\begin{tabular}{|c|c|c|c|c|}
\hline
Type & $f(x,y,z)$ & $x$ & $y$ & $z$ \\
\hline
I & $-\prod_{i=3}^{m+2}(y^{\frac{p_2}{m}}-\lambda_i x^{\frac{p_1}{m}}) +z^{p_3}$ & $X_1$ & $X_2$ & $X_3 \cdots X_{m+2}$  \\
${\rm II}_1$ & $-x\prod_{i=3}^{m+2}(y^{\frac{p_2}{m}}-\lambda_i x^{\frac{p_1-1}{m}}) +z^{p_3}$ & $X_1^{p_3}$ & $X_2$ & $X_1X_3 \cdots X_{m+2}$  \\
${\rm II}_2$ & $-\prod_{i=3}^{m+2}(z^{\frac{p_3}{m}}-\lambda_i x^{\frac{p_1}{m}}) +xy^{p_2}$ & $X_1^{p_3}$ & $X_3 \cdots X_{m+2}$ & $X_1X_2$ \\
III & $-xy\prod_{i=3}^{m+2}(y^{\frac{p_2-1}{m}}-\lambda_i x^{\frac{p_1-1}{m}}) +z^{p_3}$ & $X_1^{p_3}$ & $X_2^{p_3}$ & $X_1X_2X_3 \cdots X_{m+2}$  \\
IV & $-x\prod_{i=3}^{m+2}(y^{\frac{p_2}{m}}-\lambda_i x^{\frac{p_1-1}{m}}) +yz^{p_3}$ & $X_1^{p_3}X_2$ & $X_2^{p_1}$ & $X_1X_3 \cdots X_{m+2}$  \\
V & $x^{q_1}y+y^{q_2}z-z^{q_3}x$ & $X_2X_3^{q_2}$ & $X_3X_1^{q_3}$ & $X_1X_2^{q_3}$ \\
\hline
\end{tabular}
\begin{tabular}{|c|c|c|c|}
\hline
Type & $\alpha_1$ &  $\alpha_2$ &  $\alpha_3$  \\
\hline
I &  $\frac{p_1}{m}$ & $\frac{p_2}{m}$ & $p_3$ \\
${\rm II}_1$ &  $\frac{p_3}{m}(p_1-1)$ & $\frac{p_2}{m}$ & $p_3$ \\
${\rm II}_2$  & $\frac{p_3}{m}(p_1-1)$ & $\frac{p_3}{m}$ & $p_2$ \\
III  & $\frac{p_3}{m}(p_1-1)$ &  $\frac{p_3}{m}(p_2-1)$ & $p_3$ \\
IV & $\frac{p_3}{m}(p_1-1)$ & $\frac{p_1p_2-p_1+1}{m}$ & $p_3$ \\
V & $q_2q_3-q_3+1$ & $q_3q_1-q_1+1$ & $q_1q_2-q_2+1$ \\
\hline
\end{tabular}
\end{center}
\caption{The image of the generators in $R_{A,\bm{\lambda}}$}\label{TabGenerator}
\end{table}
It follows from the embedding $R_f \hookrightarrow R_{A,\bm{\lambda}}$ into the ring of a weighted projective line that $L_f \cong \ZZ$.

\section{Graded matrix factorizations of size two} \label{sect:mf2}
We shall now consider $L_f$-graded matrix factorizations of size two for the following two cases
\begin{itemize}
\item[(A)] $f$ is a weighted homogeneous deformation of an invertible polynomial in three variables,
\item[(B)] $f$ is an invertible polynomial in two variables.
\end{itemize}

(A)  Let $S=\CC[x,y,z]$ and let $f(x,y,z)$ be a weighted homogeneous deformation of an invertible polynomial of the form 
\[
f(x,y,z)=p_1(x,y,z)h_1(x,y,z)+p_2(x,y,z)h_2(x,y,z),
\]
where  $p_i \in S_{\vec{p}_i}$ for some $\vec{p}_i \in L_f$, $i=1,2$, and $p_1,p_2$ forms a regular sequence in $S$.
As in \cite[Section~2]{BGS} (cf.\ \cite[Definition~2.11]{AT}), we consider the Koszul resolution of the $L_f$-graded $S$-module $M=S/(p_1,p_2)$
\[
0 \longrightarrow S(-\vec{p}_1-\vec{p}_2) \longrightarrow S(-\vec{p}_1) \oplus S(-\vec{p}_2) \longrightarrow S \longrightarrow S/(p_1,p_2) \longrightarrow 0.
\]
This yields the $L_f$-graded matrix factorization $Q:=\Big(\mf{F_0}{f_0}{f_1}{F_1}\Big)$ of $f$ such that
\[
F_0=S(\vec{0}) \oplus S(\vec{f}-\vec{p}_1-\vec{p}_2), \quad F_1:=S(\vec{f}-\vec{p}_1) \oplus S(\vec{f}-\vec{p}_2),
\]
and $f_0$ and $f_1$ are given by the matrices 
\[ q_0 = \begin{pmatrix} h_1(x,y,z) & -p_2(x,y,z)\\  h_2(x,y,z) & p_1(x,y,z) \end{pmatrix} \mbox{ and }  q_1= \begin{pmatrix} p_1(x,y,z) & p_2(x,y,z) \\  -h_2(x,y,z) & h_1(x,y,z) \end{pmatrix}
\]
respectively. 

We associate to the matrix factorization a complete intersection singularity in $\CC^4$. Denote the coordinates of $\CC^4$ by $w,x,y,z$. 
\begin{definition}
The complete intersection singularity $(X_Q,0)$ is the singularity defined by 
\begin{eqnarray*} \mathbf{F}_Q(w,x,y,z) & = & (F_{Q,1}(w,x,y,z),F_{Q,2}(w,x,y,z))\\
& := &(-h_1(x,y,z)+wp_2(x,y,z), h_2(x,y,z)+wp_1(x,y,z)).
\end{eqnarray*}
\end{definition}

We can get the function $f$ back from $\mathbf{F}_Q$ by the linear reduction considered, in the case $p_1(x,y,z)=y$ and $p_2(x,y,z)=z$, by Wall \cite[7.9]{WPLMS}. Namely, let $L_w\mathbf{F}_Q$ be the elimination of the variable $w$ from the functions $F_{Q,1},F_{Q,2}$. This is given by the resultant $R_w(F_{Q,1},F_{Q,2})$ of the functions with respect to the variable $w$. We have
\[ L_w\mathbf{F}_Q=R_w(F_{Q,1},F_{Q,2}) = \det q_1 = f .
\]

 Now suppose that $p_1(x,y,z)=y$, $p_2(x,y,z)=z$, $h_2(x,y,z)=z$ and $h_1(x,y,z) = 2a(x,y)z +b(x,y)$. Then
\[
(F_{Q,1}(w,x,y,z),F_{Q,2}(w,x,y,z))=(-h_1(x,y,z)+wz, z+wy)
\]
and the substitution $z=-wy$ in $F_{Q,1}$ gives
\[ F_Q(w,x,y):=F_{Q,1}(w,x,y,-wy)=-w^2y-h_1(x,y,-wy)=-w^2y+2wya(x,y)-	b(x,y) .
\]
Thus $(X_Q,0)$ is in this case a hypersurface singularity defined by $F_Q=0$.
The discriminant of the polynomial $F_Q(w,x,y)$ in the variable $w$ is
\[ \Delta_wF_Q(w,x,y)= y^2a(x,y)^2 -yb(x,y).
\]
This is the reduction $\Delta_wF_Q$ of the polynomial $F_Q$ considered by Wall in \cite[7.3]{WPLMS}.
The polynomial $f$ is equal to
\[ f(x,y,z)=z^2 + 2ya(x,y)z +yb(x,y).
\]
Under the substitution $\widetilde{z}:=z+ya(x,y)$ this becomes
\[ f(x,y,\widetilde{z})=\widetilde{z}^2-y^2a(x,y)^2+yb(x,y) = \widetilde{z}^2 - \Delta_wF_Q(w,x,y).
\]

(B)  Now let $S=\CC[x,y]$ and let $f(x,y)$ be an invertible polynomial  of the form $f(x,y)=x^{\alpha}y^{\beta}+y^3$, where $\alpha \geq 5$. We consider the Koszul resolution of the $L_f$-graded $S$-module $M=S/(x^2,y^2)$
\[
0 \longrightarrow S(-2\vec{x}-2\vec{y}) \longrightarrow S(-2\vec{x}) \oplus S(-2\vec{y}) \longrightarrow S \longrightarrow S/(x^2,y^2) \longrightarrow 0.
\]
This yields the $L_f$-graded matrix factorization $Q:=\Big(\mf{F_0}{f_0}{f_1}{F_1}\Big)$ of $f$ such that
\[
F_0=S(\vec{0}) \oplus S(\vec{f}-2\vec{x}-2\vec{y}), \quad F_1:=S(\vec{f}-2\vec{x}) \oplus S(\vec{f}-2\vec{y}),
\]
and $f_0$ and $f_1$ are given by the matrices 
\[ q_0 = \begin{pmatrix} x^{\alpha-2}y^{\beta} & -y^2\\  y & x^2 \end{pmatrix} \mbox{ and }  q_1= \begin{pmatrix} x^2 & y^2 \\  -y& x^{\alpha-2}y^{\beta} \end{pmatrix}
\]
respectively. 

\begin{definition} We define a singularity $(X_Q,0)$ by
\[
\mathbf{F}_Q(w,x,y):= (F_{Q,1}(w,x,y),F_{Q,2}(w,x,y))= \left( x^{\alpha-2}y^{\beta} +\frac{w}{x}y^2 , y-\frac{w}{x} x^2 \right).
\]
If we substitute $y=\frac{w}{x} x^2$ in $F_{Q,1}$, we obtain the function 
\[
F_Q(w,x): = F_{Q,1}\left(w,x,\frac{w}{x} x^2 \right) = w^3x+x^{\alpha+\beta-2}w^{\beta}.
\]
Thus $(X_Q,0)$ is in fact a hypersurface singularity defined by $F_Q=0$.
\end{definition}
We get the function $f$ back from $F_Q$ by the reduction considered in  \cite[p.~19]{EW}, namely $F_Q$ is of the form
$F_Q(w,x)=w^3x+x^3A\left(w,x\right)$, where $A(w,x)=x^{\alpha+\beta-5}w^{\beta}$, and $RF_Q(w,x)= w^3+x^5A\left(\frac{w}{x},x\right)=f(w,x)$.

\section{Graded matrix factorizations of size two for regular systems of weights with $\eps=-1$}
We shall now consider the graded matrix factorizations of size two for the regular systems of weights with $\eps=-1$. They were determined in \cite{KST2}. There are 22 such systems of weights. They define the Fuchsian hypersurface singularities of genus 0 (see, e.g., \cite{ERIMS}). Let $W=(a,b,c;h)$ be such a weight system. By definition, it is reduced. The number $\eps_W$ is defined by $\eps_W := a+b+c-h$. Let $A_W=(\alpha_1, \ldots , \alpha_r)$ be the signature of the system of weights $W$. Then $(\alpha_1, \ldots , \alpha_r)$ corresponds to the signature of the corresponding Fuchsian singularity. For $r=3$, there are 14 such weight systems which correspond to the 14 exceptional unimodal singularities. Here the polynomial $f_W$ is an invertible polynomial and its canonical system of weights coincides with $(a,b,c;h)$. For $r=4$, there are 6 corresponding to the heads of the bimodal series. There are 2 more with $r=5$. The corresponding equations are weighted homogeneous deformations $f_W$ of an invertible polynomial, see Table~\ref{TabDefInv}. 
\begin{table}[h]
\begin{center}
\begin{tabular}{|c|c|c|c|c|}
\hline
$W$ & $f_W$ & Type & $p_1,p_2,p_3$ & $A=A_W$ \\
\hline
6,14,21;42 & $-(y^3-x^7)+z^2$ & I, $m=1$ & 3,7,2 & 3,7,2 \\
4,10,15;30 & $-x(y^5-x^2)+z^2$ & ${\rm II}_1, m=1$ & 3,5,2 & 4,5,2 \\
3,8,12;24 & $-x(y^4-x)+z^3$ & ${\rm II}_1, m=1$ & 2,4,3 & 3,4,3 \\
6,8,15;30 & $-x(y^3-x^4)+z^2$ & ${\rm II}_1, m=1$ & 5,3,2 & 8,3,2 \\
4,6,11;22 & $-xy(y^2-x^3)+z^2$ & III, $m=1$ & 4,3,2 & 6,4,2 \\
3,5,9;18 & $-x(y^3-x)+yz^3$ & IV, $m=1$ & 2,3,3 & 3,5,3 \\
4,5,10;20 & $-x(y^2-x)+z^5$ & ${\rm II}_1, m=1$ & 2,2,5 & 5,2,5 \\
3,4,8;16 & $-x(y^2-x)+yz^4$ & IV, $m=1$ & 2,2,4 & 4,3,4 \\
6,8,9;24 & $-x(y^2-x^3)+z^3$ & ${\rm II}_1, m=1$ & 4,2,3 & 9,2,3 \\
4,6,7;18 & $-x(y^3-x^2)+yz^2$ & IV, $m=1$ & 3,3,2 & 4,7,2 \\
3,5,6;15 & $-xy(y-x^2)+z^3$ & III, $m=1$ & 3,2,3 & 6,3,3 \\
4,5,6;16 & $-x(y^2-x^3)+yz^2$ & IV, $m=1$ & 4,2,2 & 6,5,2 \\
3,4,5;13 & $-z(xz-y^2)+yx^3$ & V & 3,2,2 & 3,4,5 \\
3,4,4;12 & $-xy(y-x)+z^4$ & III, $m=1$ & 2,2,4 & 4,4,4\\
2,6,9;18 & $-x(y^3-x)(y^3-\lambda_4 x)+z^2$&  ${\rm II}_1$, $m=2$ & 3,6,2 & 2,3,2,2 \\
2,4,7;14 & $-xy(y-x^2)(y-\lambda_4 x^2)+z^2$ & III, $m=2$ & 5,3,2 & 4,2,2,2 \\
2,4,5;12 & $-x(y^2-x)(y^2-\lambda_4 x) +yz^2$ & IV, $m=2$ & 3,4,2 & 2,5,2,2 \\
2,3,6;12 & $\begin{array}{c} -(z^3-x)(z^3-\lambda_4 x)+xy^2 \\ -(z^2-x)(z^2-\lambda_4 x)+xy^3 \end{array} $ & ${\rm II}_2, m=2$& $\begin{array}{c} 2,2,6\\ 2,3,4 \end{array}$ & $\begin{array}{c} 3,3,2,2 \\ 2,2,3,3 \end{array}$\\
2,3,4;10 & $-x(y-x^2)(y- \lambda_4 x^2)+yz^2$ & IV, $m=2$ & 5,2,2 & 4,3,2,2 \\
2,3,3;9  & $-x(y-x)(y-\lambda_4 x)+yz^3$ & IV, $m=2$ & 3,2,3 & 3,2,3,3 \\
2,2,5;10 & $-xy(y-x)(y-\lambda_4 x)(y-\lambda_5 x) +z^2$ & III, $m=3$ & 4,4,2 & 2,2,2,2,2 \\
2,2,3;8 & $-x(y-x)(y-\lambda_4 x)(y-\lambda_5 x) +yz^2$ & IV, $m=3$ & 4,3,2 & 2,3,2,2,2 \\
\hline
\end{tabular}
\end{center}
\caption{Type of $f_W$}\label{TabDefInv}
\end{table}
This again implies that $L_{f_W} \cong \ZZ$.
Thus all the regular systems of weights with $\eps_W=-1$ can be given by weighted homogeneous deformations $f$ of invertible polynomials with $\eps_f=-1$ and $L_f \cong \ZZ$ and the $L_f$-graded matrix factorizations of $f$ correspond to the graded matrix factorizations considered in \cite{KST2}. The Dolgachev numbers are equal to the signature of the system of weights.

It will turn out that the graded matrix factorizations of size 2 correspond to some of these singularities, but also to the Fuchsian complete intersection singularities of genus 0 (see \cite{ERIMS}). There are 8+5+2 of them for certain signatures $(\alpha_1, \ldots , \alpha_r)$. 

Let $W$ be a regular system of weights with $\eps_W=-1$ and signature $A_W=(\alpha_1, \ldots , \alpha_r)$. In \cite{KST2}, a full strongly exceptional collection in the triangulated category of graded matrix factorizations associated to $W$ was constructed.  The full strongly exceptional collection is described by the graph shown in Fig.~\ref{FigT+pqr}.
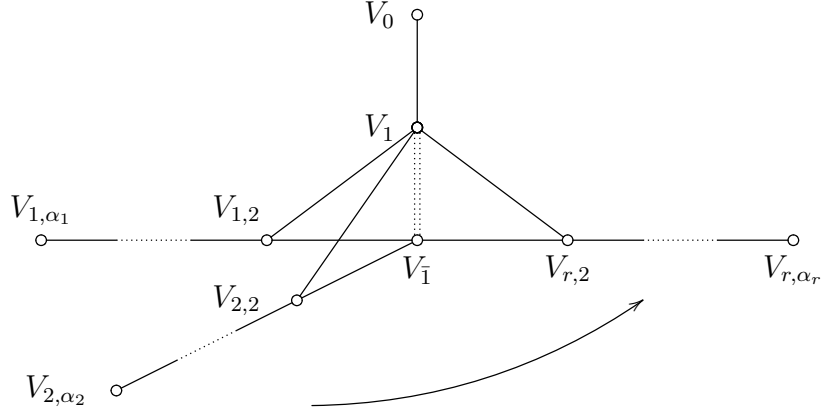
\begin{figure}
\[
\begin{xy}
(-5, 30)*{V_0}, (-5,15)*{V_1},
(0,-4)*{V_{\bar{1}}}, (20,-4)*{V_{r,2}},  (50,-4)*{V_{r,\alpha_r}},
(-24,4)*{V_{1,2}},  (-50,4)*{V_{1,\alpha_1}},
(-24,-8)*{V_{2,2}}, (-48,-20)*{V_{2,\alpha_2}},
\ar@{-} (0,30)*\cir<2pt>{}; (0,15)*\cir<2pt>{}="A-1"
\ar@{:} "A-1"; (0,0)*\cir<2pt>{}="A"
\ar@{-}  "A" ;  (20,0) *\cir<2pt>{}="B1"
\ar@{-}  "B1"; (30,0) \ar@{.} (30,0); (40,0)
\ar@{-}  (40,0);  (50,0)  *\cir<2pt>{}
\ar@{-}  "A";  (-20,0)  *\cir<2pt>{}="B2"
\ar@{-}  "B2"; (-30,0) \ar@{.} (-30,0); (-40,0)
\ar@{-}  (-40,0);  (-50,0)  *\cir<2pt>{}
\ar@{-}  "A";  (-16,-8)  *\cir<2pt>{}="B3"
\ar@{-}  "B3"; (-24,-12) \ar@{.} (-24,-12); (-32,-16)
\ar@{-}  (-32,-16);  (-40,-20)  *\cir<2pt>{}
\ar@/_10pt/ (-14,-22) ; (30,-8)
\ar@{-} (0,15)*\cir<2pt>{}; "B1"
\ar@{-} (0,15)*\cir<2pt>{}; "B2"
\ar@{-} (0,15)*\cir<2pt>{}; "B3"
\end{xy}
\]
\caption{The graph $T^+_{\alpha_1,\ldots ,\alpha_r}$} \label{FigT+pqr}
\end{figure}

It turns out that a graded matrix factorization of size 2 corresponds to the end point $V_{i,\alpha_i}$ of some arm.
\begin{theorem} \label{thm:main}
Let $f$ be a weighted homogeneous deformation of an invertible polynomial  with $\eps_f=-1$ and $L_f \cong \ZZ$. Let $A=(\alpha_1, \ldots , \alpha_r)$ be the Dolgachev numbers of $f$. If the $L_f$-graded matrix factorization $V_{i,\alpha_i}$ is of rank 2, then the corresponding singularity $\mathbf{F}_{V_{i,\alpha_i}}$ has the signature $(\widetilde{\alpha}_1, \ldots , \widetilde{\alpha}_r)$ with $\widetilde{\alpha}_i=\alpha_i+1$ and $\widetilde{\alpha}_j=\alpha_j$ for $j \neq i$.
\end{theorem}

\begin{proof} We consider the matrix factorizations for the cases of Table~\ref{TabGenerator} and examine conditions under which there exist maps of the local ring of the corresponding complete intersection singularity to the ring of a weighted projective line of type $(\widetilde{\alpha}_1, \ldots , \widetilde{\alpha}_r)$. We indicate the matrix factorization, the mapping, the index $i$ where we have an increase of $\alpha_i$ by one, and the condition under which the mapping gives an embedding. 

(I) $m \geq 2$, $m|p_1$, $m|p_2$. We consider the matrix factorization
\[ 
q_0= \begin{pmatrix} \prod_{i=4}^{m+2}(y^{\frac{p_2}{m}}-\lambda_i x^{\frac{p_1}{m}}) & -z \\  z^{p_3-1} &  -(y^{\frac{p_2}{m}}-x^{\frac{p_1}{m}}) \end{pmatrix}.
\]
Mapping:
\[
x=X_1X_3^{\frac{p_2}{m}}, \ y=X_2X_3^{\frac{p_1}{m}}, \ z=X_3^{(m-1)\frac{p_1p_2}{m^2}}X_4 \cdots X_{m+2}, \ w=X_4^{p_3-1} \cdots X_{m+2}^{p_3-1}
\]
Increase of $\alpha_3$ under the condition:
\begin{equation}
(m-1)\frac{p_1p_2}{m^2}(p_3-1)=\frac{p_1p_2}{m^2}+p_3+1. \label{eq:I}
\end{equation}

(${\rm II}_1$) $m \geq 1$, $m|(p_1-1)$, $m|p_2$. We have two essentially different matrix factorizations:

(a)
\[ 
q_0= \begin{pmatrix} \prod_{i=3}^{m+2}(y^{\frac{p_2}{m}}-\lambda_i x^{\frac{p_1-1}{m}}) & -z \\  z^{p_3-1} &  -x \end{pmatrix}.
\]
Mapping:
\[
x=X_1^{p_3+ \frac{m+p_2}{p_1-1}}, \ y=X_1X_2, \ z=X_1^{p_2}X_3 \cdots X_{m+2}, \ w=X_3^{p_3-1} \cdots X_{m+2}^{p_3-1}
\] 
Increase of $\alpha_1$ under the condition:
\begin{equation}
p_2(p_3-1)(p_1-1)=p_3(p_1-1)+m+p_2.
\end{equation}

(b) 
\[ 
q_0= \begin{pmatrix} x\prod_{i=4}^{m+2}(y^{\frac{p_2}{m}}-\lambda_i x^{\frac{p_1-1}{m}}) & -z \\  z^{p_3-1} &  -(y^{\frac{p_2}{m}}-x^{\frac{p_1-1}{m}}) \end{pmatrix}.
\]
Mapping:
\begin{eqnarray*}
x=X_1^{p_3}X_3^{\frac{p_2}{m}}, & &  y=X_2X_3^{\frac{p_1-1}{m}},  \quad z=X_1X_3^{\frac{p_2}{m}+(m-1)\frac{(p_1-1)p_2}{m^2}}X_4 \cdots X_{m+2}, \\
& &  w=X_1^{p_3-1}X_4^{p_3-1} \cdots X_{m+2}^{p_3-1}
\end{eqnarray*}
Increase of $\alpha_3$ under the condition:
\begin{equation}
\left( \frac{p_2}{m} + (m-1)\frac{(p_1-1)p_2}{m^2} \right) (p_3-1)=\frac{(p_1-1)p_2}{m^2}+p_3+1.
\end{equation}

(${\rm II}_2$) $m \geq 2$, $m|p_1$, $m|p_3$. We consider the matrix factorization
\[ 
q_0= \begin{pmatrix} \prod_{i=4}^{m+2}(z^{\frac{p_3}{m}}-\lambda_i x^{\frac{p_1}{m}}) & -y \\  xy^{p_2-1} &  -(z^{\frac{p_3}{m}}-x^{\frac{p_1}{m}}) \end{pmatrix}.
\]
Mapping:
\begin{eqnarray*} 
x=X_1^{p_3}X_3^{\frac{p_3}{m}}, & &  y=X_3^{(m-1)\frac{p_1p_3}{m^2}}X_4 \cdots X_{m+2}, \quad  z=X_1X_2X_3^{\frac{p_1}{m}}, \\
& & w=X_1^{(m-1)\frac{p_3}{m}}X_4^{p_2-1} \cdots X_{m+2}^{p_2-1}
\end{eqnarray*}
Increase of $\alpha_3$ under the condition:
\begin{equation}
\frac{p_3}{m}+(m-1)\frac{p_1p_3}{m^2}(p_2-1)=(m-1)\frac{p_1p_3}{m^2}+p_2+1.
\end{equation}

One can show that the other possible matrix factorization
\[
q_0= \begin{pmatrix} \prod_{i=4}^{m+2}(z^{\frac{p_3}{m}}-\lambda_i x^{\frac{p_1}{m}}) & -x \\  y^{p_2} &  -(z^{\frac{p_3}{m}}-x^{\frac{p_1}{m}}) \end{pmatrix}
\]
has no solution.

(III) $m \geq 1$, $m|(p_1-1)$, $m|(p_2-1)$. We have the essentially different matrix factorizations:

(a) 
\[ 
q_0= \begin{pmatrix} y\prod_{i=3}^{m+2}(y^{\frac{p_2-1}{m}}-\lambda_i x^{\frac{p_1-1}{m}}) & -z \\  z^{p_3-1} &  -x \end{pmatrix}.
\]
Mapping:
\[
x=X_1^{p_3+ \frac{(m-1)+p_2}{p_1-1}}, \ y=X_1X_2^{p_3}, \ z=X_1^{p_2}X_2X_3\cdots X_{m+2}, \ w=X_2^{p_3-1}X_3^{p_3-1} \cdots X_{m+2}^{p_3-1}
\] 
Increase of $\alpha_1$ under the condition:
\begin{equation}
p_2(p_3-1)(p_1-1)=p_3(p_1-1)+(m-1)+p_2.
\end{equation}

(b)
\[ 
q_0= \begin{pmatrix} xy\prod_{i=4}^{m+2}(y^{\frac{p_2-1}{m}}-\lambda_i x^{\frac{p_1-1}{m}}) & -z \\  z^{p_3-1} &  -(y^{\frac{p_2-1}{m}}-x^{\frac{p_1-1}{m}}) \end{pmatrix}.
\]
Mapping:
\begin{eqnarray*}
x=X_1^{p_3}X_3^{\frac{p_2-1}{m}}, & &  y=X_2^{p_3}X_3^{\frac{p_1-1}{m}},  \quad z=X_1X_2X_3^{\frac{p_1-1}{m}+\frac{p_2-1}{m}+(m-1)\frac{(p_1-1)(p_2-1)}{m^2}}X_4 \cdots X_{m+2}, \\
& &  w=X_1^{p_3-1}X_2^{p_3-1}X_4^{p_3-1} \cdots X_{m+2}^{p_3-1}
\end{eqnarray*}
Increase of $\alpha_3$ under the condition:
\begin{equation}
\left( \frac{p_1-1}{m}+\frac{p_2-1}{m} + (m-1)\frac{(p_1-1)(p_2-1)}{m^2} \right) (p_3-1)=\frac{(p_1-1)(p_2-1)}{m^2}+p_3+1.
\end{equation}

(IV) $m \geq 1$, $m|(p_1-1)$, $m|p_2$. We have three essentially different factorizations.

(a)
\[ 
q_0= \begin{pmatrix} \prod_{i=3}^{m+2}(y^{\frac{p_2}{m}}-\lambda_i x^{\frac{p_1-1}{m}}) & -z \\  yz^{p_3-1} &  -x \end{pmatrix}.
\]
Mapping:
\[
x=X_1^{p_3+ \frac{m+p_2}{p_1-1}}X_2, \ y=X_1X_2^{p_1}, \ z=X_1^{p_2}X_3 \cdots X_{m+2}, \ w=X_2^{p_1-1}X_3^{p_3-1} \cdots X_{m+2}^{p_3-1}
\] 
Increase of $\alpha_1$ under the condition:
\begin{equation}
p_2(p_3-1)(p_1-1)=p_3(p_1-1)+(m+1-p_1)+p_2. \label{eq:IVa}
\end{equation}

(b)
\[ 
q_0= \begin{pmatrix} \prod_{i=3}^{m+2}(y^{\frac{p_2}{m}}-\lambda_i x^{\frac{p_1-1}{m}}) & -y \\  z^{p_3} &  -x \end{pmatrix}.
\]
Mapping:
\[
x=X_1^{p_3}X_2^{p_3}, \  y=X_2^{p_1+\frac{(p_1-1)(p_3-1)+m}{p_2}},  \ z=X_1X_2X_3 \cdots X_{m+2}, \ w=X_3^{p_3} \cdots X_{m+2}^{p_3}
\]
Increase of $\alpha_2$ under the condition:
\begin{equation}
p_2(p_1(p_3-1)-p_3)=(p_1-1)(p_3-1)+m. \label{eq:IVb}
\end{equation}
This equation is equivalent to Equation~(\ref{eq:IVa}).

(c)
\[ 
q_0= \begin{pmatrix} x\prod_{i=4}^{m+2}(y^{\frac{p_2}{m}}-\lambda_i x^{\frac{p_1-1}{m}}) & -z \\  yz^{p_3-1} &  -(y^{\frac{p_2}{m}}-x^{\frac{p_1-1}{m}}) \end{pmatrix}.
\]
Mapping:
\begin{eqnarray*}
x=X_1^{p_3}X_2X_3^{\frac{p_2}{m}}, & &  y=X_2^{p_1}X_3^{\frac{p_1-1}{m}},  \quad z=X_1X_3^{\frac{p_2}{m}+(m-1)\frac{(p_1-1)p_2}{m^2}}X_4 \cdots X_{m+2}, \\
& &  w=X_1^{p_3-1}X_2^{1+(m-1)\frac{p_1-1}{m}}X_4^{p_3-1} \cdots X_{m+2}^{p_3-1}
\end{eqnarray*}
Increase of $\alpha_3$ under the condition:
\begin{equation}
\left( \frac{p_2}{m} + (m-1)\frac{(p_1-1)p_2}{m^2} \right) (p_3-1)+\frac{p_1-1}{m}=\frac{(p_1-1)p_2}{m^2}+p_3+1. \label{eq:IVc}
\end{equation}

Similarly as in $({\rm II}_2)$, one can show that the other matrix factorization
\[ 
q_0= \begin{pmatrix} x\prod_{i=4}^{m+2}(y^{\frac{p_2}{m}}-\lambda_i x^{\frac{p_1-1}{m}}) & -y \\  z^{p_3} &  -(y^{\frac{p_2}{m}}-x^{\frac{p_1-1}{m}}) \end{pmatrix}.
\]
has no solution.

(d) We also have to consider a special matrix factorization in the case $m=1$. We can write the equation as
\[
f(x,y,z)=-y(xy^{p_2-1}-z^{p_3})+x^{p_1}.
\]
Then we have the matrix factorization 
\[ 
q_0= \begin{pmatrix} xy^{p_2-1}-z^{p_3} & -x \\  x^{p_1-1} &  -y \end{pmatrix}
\]
Mapping:
\[ 
x=X_2^{p_3}X_3^{p_3}, \quad y=x_2^{p_1+\frac{2}{p_2-1}}, \quad z=X_1X_2X_3, \quad w=X_3^{p_3(p_1-1)}
\]
Increase of $\alpha_2$ under the condition:
\begin{equation}
p_3(p_1-1)(p_2-1)=p_1(p_2-1)+2.
\end{equation}
This equation is equivalent to Equation~(\ref{eq:IVb}) with $m=1$.

(V) The equation is
\[
f(x,y,z)=-z(xz^{p_3-1}-y^{p_2})+yx^{p_1}.
\]
Then we have two matrix factorizations:

(a)
\[
q_0= \begin{pmatrix} xz^{p_3-1}-y^{p_2} & -x \\  yx^{p_1-1} &  -z \end{pmatrix}
\]
Mapping:
\[ 
x=X_2^{p_2}X_3^{p_3}, \quad y=X_1^{p_3}X_2X_3, \quad z=X_1X_2^{p_1p_2-p_2+1}, \quad w=X_1^{p_3-1}X_3^{p_1p_2-p_2+1}
\]
Increase of $\alpha_2$ under the condition:
\begin{equation}
p_2(p_1-1)(p_3-1)=(p_1-1)(p_3-1)+2. \label{eq:Va}
\end{equation}

(b)
\[
q_0= \begin{pmatrix} xz^{p_3-1}-y^{p_2} & -y\\  x^{p_1} &  -z \end{pmatrix}
\]
Mapping:
\[
x=X_1X_2X_3^{p_2}, \quad y=X_1^{p_1p_3-p_1+1}X_3, \quad z=X_1^{p_1}X_2^{p_1}, \quad w=X_2^{p_1}X_3^{p_1p_2}
\]
Increase of $\alpha_1$ under the condition~(\ref{eq:Va}).

In the cases I, ${\rm II}_1$(b), ${\rm II}_2$, III(b), and IV(c), one can exchange the factor with $\lambda_3(=1)$ in the lower right corner of the matrix factorization by any of the factors corresponding to $\lambda_i$, $i=4, \ldots, m+2$, and interchange $X_3$ and $X_i$. This leads to an increase of $\alpha_i$ instead of $\alpha_3$. In Case III(a), one can interchange the variables $x$ and $y$ to get an increase of $\alpha_2$.

The solutions to the equations~(\ref{eq:I})--(\ref{eq:Va}) are precisely the entries of Table~\ref{TabDefInv} and two extra ones. Equation~(\ref{eq:I}) has the solution I, $m=2$, $(p_1,p_2,p_3)=(4,4,3)$, $A=(2,2,3,3)$. This corresponds to the minimally elliptic singularity $V'_{(1)}$ of \cite{EPLMS} with weight system $W=(3,3,4;12)$ and $\eps_W=-2$. Equation~(\ref{eq:IVc}) has the additional solution IV, $m=1$, $(p_1,p_2,p_3)=(3,3,4)$, $A=(8,7,4)$. This corresponds to a singularity with the weight system $W=(12,8,7;36)$ with $\eps_W=-9$.
\end{proof}

\begin{remark} The proof shows that the statement of Theorem~\ref{thm:main} only holds in this very special situation. Namely, it shows that, except for two exceptional cases, the weighted homogeneous deformations $f$ of an invertible polynomial with $\eps_f=-1$ are the only weighted homogeneous deformations of an invertible polynomial with the property that there exists an $L_f$-graded matrix factorization $Q$ of size 2 such that the corresponding singularity $\mathbf{F}_Q$ has the signature $(\widetilde{\alpha}_1, \ldots , \widetilde{\alpha}_r)$ with $\widetilde{\alpha}_i=\alpha_i+1$ and $\widetilde{\alpha}_j=\alpha_j$ for $j \neq i$ where $A=(\alpha_1, \ldots , \alpha_r)$ are the Dolgachev numbers of $f$. 
\end{remark}

\begin{remark} \label{rem:R}
The reduction is a relation between the singularities. Not all possible extensions of the signatures are realised. For example, the following pairs are missing: $(2,3,7) \not\leftarrow (2,3,8)$, $(2,4,5) \not\leftarrow (2,4,6)$, $(3,3,4) \not\leftarrow (3,3,5)$, $(2,2,2,3) \not\leftarrow (2,2,2,4)$. In these cases, the graded matrix factorization $V_{r,\alpha_r}$ for the singularities with signature $(2,3,7)$, $(2,4,5)$, $(3,3,4)$, and $(2,2,2,3)$ is of rank 4. In these cases, the function $f_W$ can be given by
$f_W(x,y,z)=f'_W(x,y)+z^2$, where $f'_W(x,y)$ is  the invertible polynomial in two variables given by
\[
(2,3,7): x^7+y^3 \quad (2,4,5):x^5y+y^3 \quad (3,3,4):x^8+y^3 \quad(2,2,2,3):x^6y+y^3.
\]
We shall consider the $L_f$-graded matrix factorization of size 2 of $f'_W$ and the corresponding reduction $R$ defined in Section~\ref{sect:mf2}(B).
We enhance the reduction by these cases.
\end{remark}

\begin{example} We consider the case $W=(4,10,15;30)$, $A_W=(2,4,5)$, $f_W=x^5y+y^3+z^2$, $f'_W=x^5y+y^3$. This is the singularity $E_{13}$ in Arnold's notation. We consider the graded matrix factorization of $f'_W$ given by
\[
q_0=\begin{pmatrix} x^3y & -y^2  \\  y &x^2  \end{pmatrix}.
\]
Then $\mathbf{F}_Q(w,x)=w^3x+wx^4$. This is the singularity $Z_{12}$ with signature $(2,4,6)$.
\end{example}

The isolated hypersurface singularities defined by regular systems of weights with $\eps=-1$ and $r \leq 4$ and the isolated complete intersection singularities which reduce to them are Kodaira singularities \cite{EW} or minimally elliptic singularities \cite{Laufer77}. The minimal resolution is of Kodaira type II, III, IV or ${\rm I}_0^\ast$ (see \cite{EW}) and it carries a fundamental cycle $Z$ (see \cite[p.~132]{Artin}). The self-intersection number of the fundamental cycle $D:=-Z^2$ is called the {\em grade} of the singularity \cite{EW}. Let $(X,0)$ be any of the singularities of Table~\ref{TabDefInv} with Dolgachev numbers $\alpha_1, \ldots, \alpha_r$ and $r \leq 4$ or of the isolated complete intersection singularities which reduce to them.  According to \cite{EW} or \cite{Laufer77},  there is the correspondence between certain types of Dolgachev numbers,  Kodaira types, the number $\ell$ of parameters $k_i$ of the Kodaira type,  and the grades indicated in Table~\ref{TabKodaira} (see also \cite[p.~515]{Saito}).
\begin{table}[h]
\begin{center}
\begin{tabular}{|c|c|c|c|}
\hline
Type of Dolgachev numbers & Kodaira type & $\ell$ & $D=-Z^2$\\
\hline
$2,3,c$; $c \geq 7$ & ${\rm II}(k)$ & 1  & $c-6$ \\
$2,b,c$; $b,c \geq 4$ & ${\rm III}(k_1,k_2)$ & 2 &  $b+c-8$\\
$a,b,c$; $a,b,c \geq 3$ & ${\rm IV}(k_1,k_2,k_3)$ & 3 & $a+b+c-9$ \\
$a,b,c,d$; $a,b,c,d \geq 2$ &  ${\rm I}_0^\ast(k_1,k_2,k_3,k_4)$ & 4 & $a+b+c+d-8$\\
\hline
\end{tabular}
\end{center}
\caption{Dolgachev numbers, Kodaira types, and grades} \label{TabKodaira}
\end{table}
By \cite[Corollary~7.9.2]{WPLMS}, the Kodaira type and hence the number $\ell$  of a singularity and its reduction are the same.

The possible reductions between these singularities are visualized in Fig.~\ref{FigPyramidR}. Here a usual arrow refers to a reduction corresponding to Theorem~\ref{thm:main}, a dashed arrow to the special reduction $R$ described in Section~\ref{sect:mf2}(B). The Dolgachev numbers of the regular systems of weights for $\eps=-1$ are denoted by usual letters, those of the remaining singularities are denoted by scriptstyle letters. We also indicate the names of Arnold for the hypersurface singularities \cite{AGV85} and of Wall \cite{WPLMS} (see also \cite{ERIMS}) for the complete intersection singularities. 
\begin{figure}[h]
$$
\xymatrix{  
& & & {237}  \ar@{<--}[r] \ar@{}^{E_{12}}[d]  & {238} \ar@{<-}[r] \ar@{}^{Z_{11}}[d]  & {239} \ar@{<-}[r] \ar@{}^{Q_{10}}[d] &  {\scriptstyle{2310}} \ar@{}^{J'_{9}}[d] \\
& & & {245}   \ar@{<--}[r]  \ar@{}^{E_{13}}[d]  & {246} \ar@{<-}[r]  \ar@{}^{Z_{12}}[d] & {247} \ar@{<-}[r]  \ar@{}^{Q_{11}}[d] &  {\scriptstyle{248}}  \ar@{}^{J'_{10}}[d]  \\
& & {255}  \ar@{->}[ur] \ar@{<-}[r]  \ar@{}^{W_{12}}[d] & {256} \ar@{->}[ur]  \ar@{}^{S_{11}}[d]  & {\scriptstyle{257}}  \ar@/_/[l]  \ar@{->}[ur]  \ar@{}^{L_{10}}[d] &  {\scriptstyle{266}} \ar@/^/[ll]  \ar@{}^{K'_{10}}[d] &  \\
& & & {334}   \ar@{<--}[r]  \ar@{}^{E_{14}}[d]  & {335} \ar@{<-}[r]  \ar@{}^{Z_{13}}[d]   & {336} \ar@{<-}[r]  \ar@{}^{Q_{12}}[d] &  {\scriptstyle{337}}  \ar@{}^{J'_{11}}[d]\\
& & {344}  \ar@{->}[ur] \ar@{<-}[r]  \ar@{}^{W_{13}}[d]   & {345} \ar@{->}[ur]   \ar@{}^{S_{12}}[d]  & {\scriptstyle{346}}  \ar@/_/[l] \ar@{->}[ur]  \ar@{}^{L_{11}}[d] & {\scriptstyle{355}} \ar@/^/[ll]  \ar@{}^{K'_{11}}[d] &  \\
& {444} \ar@{->}[ur] \ar@{<-}[r]  \ar@{}^{U_{12}}[d]  & {\scriptstyle{445}} \ar@{->}[ur]  \ar@{}^{M_{11}}[d]  &  {2223}  \ar@{<--}[r] \ar@{}^{J_{3,0}}[d]  & {2224}  \ar@{<-}[r] \ar@{}^{Z_{1,0}}[d]  & {2225}  \ar@{<-}[r] \ar@{}^{Q_{2,0}}[d]  & {\scriptstyle{2226}} \ar@{}^{J'_{2,0}}[d] \\
& &  {2233}  \ar@{->}[ur] \ar@{<-}[r]  \ar@{}^{W_{1,0}}[d]  & {2234} \ar@{->}[ur]  \ar@{}^{S_{1,0}}[d] & {\scriptstyle{2235}}  \ar@/_/[l] \ar@{->}[ur]  \ar@{}^{L_{1,0}}[d]  & {\scriptstyle{2344}} \ar@/^/[ll]  \ar@{}^{K'_{1,0}}[d] &  \\
& {2333} \ar@{->}[ur]  \ar@{<-}[r] \ar@{}^{U_{1,0}}[d]  &  {\scriptstyle{2334}}  \ar@{->}[ur]  \ar@{}^{M_{1,0}}[d] & & &  & \\
{\scriptstyle{3333}} \ar@{->}[ur]   \ar@{}^{I_{1,0}}[d] &  &   & &  & &  \\
&&&  &  & & 
 }
$$
\caption{The pyramid of reductions} \label{FigPyramidR}
\end{figure}
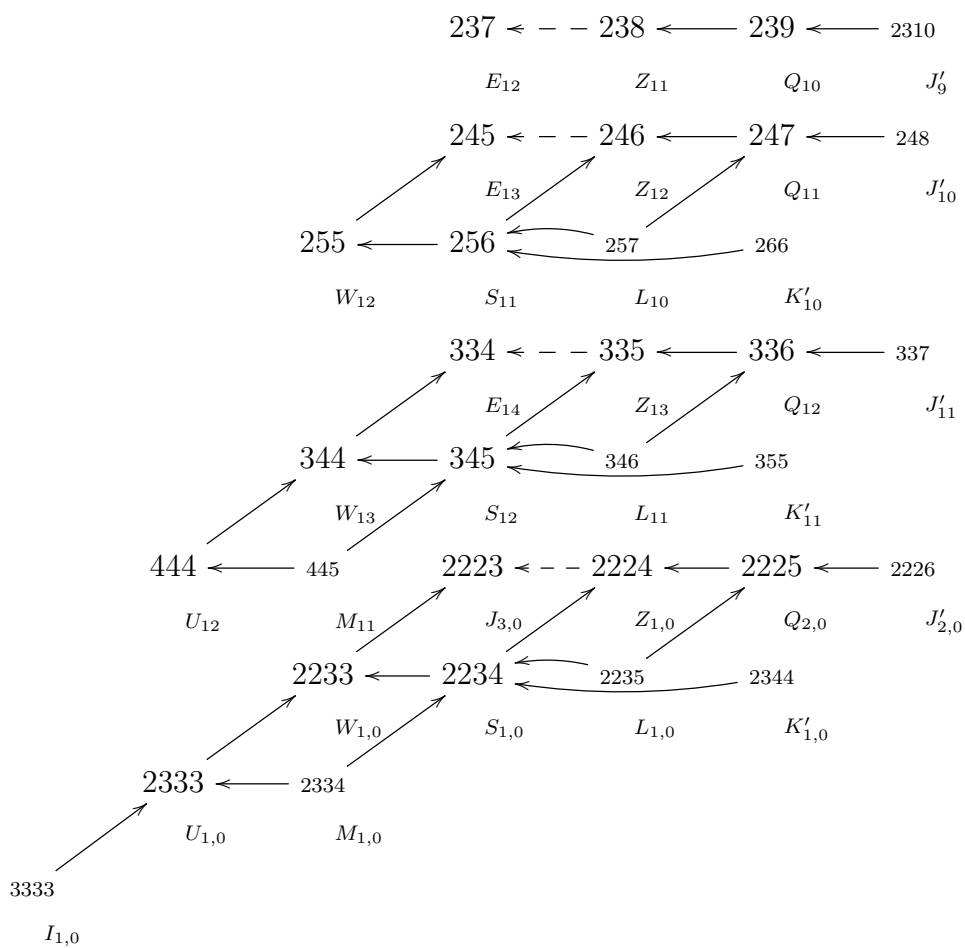

On the other hand, each of these singularities has another triple or quadruple of numbers, the {\em Gabrielov numbers}, see \cite{ET} for the definition in the case of the hypersurface singularities and \cite{ET2} in the case of the complete intersection singularities. They also have types like the Dolgachev numbers, but for them the columns corresponding to $\ell$ and $D$ are interchanged, see Table~\ref{TabKodairaG}. Note that the grade $D$ corresponds to the type of Gabrielov numbers, see  \cite[Theorem~3.13]{Laufer77}. More precisely, according to \cite[Theorem~3.13]{Laufer77}, for  $1 \leq D \leq 3$, the singularity is a hypersurface singularity and, for $D=4$, it is a complete intersection singularity. For $D=1$, the singularity is given by an equation of the form $z^2+y^3+g(x,y,z)=0$, where the polynomial $g(x,y,z)$ has only terms of degree $\geq 4$, see also \cite[Table~1]{Laufer77}. For $D=2$, the singularity is given by an equation of the form $z^2+h(x,y)=0$, where the polynomial $h(x,y)$ has only terms of degree $\geq 4$, see also \cite[Table~2]{Laufer77}. Finally, for $D=3$, the equation is $f(x,y,z)=0$, where the polynomial $f(x,y,z)$ has only terms of degree $\geq 3$. 
\begin{table}[h]
\begin{center}
\begin{tabular}{|c|c|c|}
\hline
Type of Gabrielov numbers  & $D=-Z^2$ & $\ell$\\
\hline
$2,3,c$; $c \geq 7$  & 1  & $c-6$ \\
$2,b,c$; $b,c \geq 4$  & 2 &  $b+c-8$\\
$a,b,c$; $a,b,c \geq 3$  & 3 & $a+b+c-9$ \\
$a,b,c,d$; $a,b,c,d \geq 2$  & 4 & $a+b+c+d-8$\\
\hline
\end{tabular}
\end{center}
\caption{Gabrielov numbers, grades, and numbers $\ell$} \label{TabKodairaG}
\end{table}

The hypersurface singularities with $r=3$ are the 14 exceptional unimodal singularities. There is Arnold's strange duality \cite{AGV85} between these singularities which interchanges Dolgachev and Gabrielov numbers. 

There is another partial order between singularities, namely the adjacency. We consider a special type of adjacency, namely the {\em s-adjacency} of singularities as defined by H.~Laufer in \cite[Definition~4.12]{Laufer79}.  The s-adjacency corresponds to a deformation with constant grade. By \cite[Theorem~4.13]{Laufer79}, the s-adjacencies between the 14 exceptional unimodal singularities are given by Fig.~\ref{FigPyramidAd}, where the numbers correspond to the Gabrielov numbers. This is the pyramid of adjacencies given in \cite[p.~255]{AGV85}. 
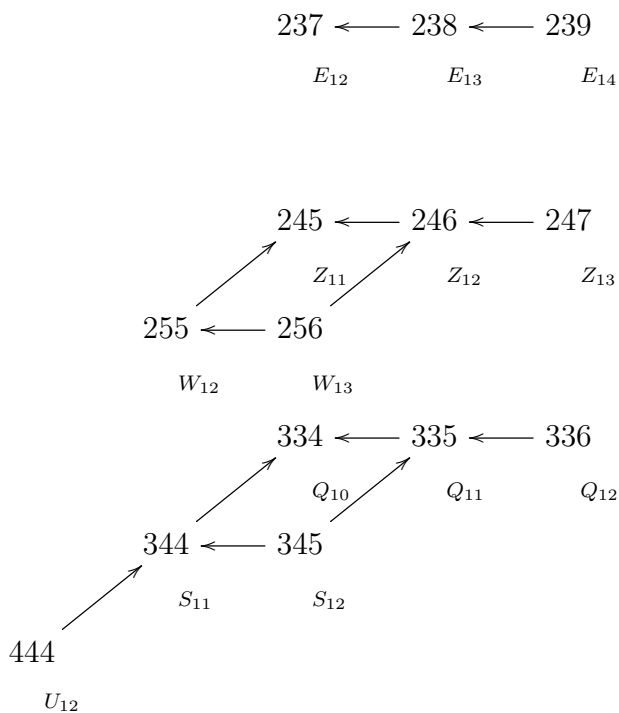
\begin{figure}[h]
$$
\xymatrix{  
 & & {237}  \ar@{<-}[r] \ar@{}^{E_{12}}[d]  & {238} \ar@{<-}[r] \ar@{}^{E_{13}}[d]  & {239}  \ar@{}^{E_{14}}[d]    \\
 & & & & \\
 & & {245}   \ar@{<-}[r]  \ar@{}^{Z_{11}}[d]  & {246} \ar@{<-}[r]  \ar@{}^{Z_{12}}[d] & {247}  \ar@{}^{Z_{13}}[d]     \\
 & {255}  \ar@{->}[ur] \ar@{<-}[r]  \ar@{}^{W_{12}}[d] & {256}   \ar@{->}[ur]  \ar@{}^{W_{13}}[d]  &  &     \\
 & & {334}   \ar@{<-}[r]  \ar@{}^{Q_{10}}[d]  & {335} \ar@{<-}[r]  \ar@{}^{Q_{11}}[d]   & {336}   \ar@{}^{Q_{12}}[d]  \\
 & {344}  \ar@{->}[ur] \ar@{<-}[r]  \ar@{}^{S_{11}}[d]   & {345} \ar@{->}[ur]   \ar@{}^{S_{12}}[d]  &  &   \\
 {444} \ar@{->}[ur]   \ar@{}^{U_{12}}[d]  &   &    &   &    \\
&&  &  &  
 }
$$
\caption{The pyramid of s-adjacencies between the 14 exceptional unimodal singularities} \label{FigPyramidAd}
\end{figure}

\begin{remark}
Note that the s-adjacency is a special case of adjacency.  For example, the general adjacency between the 14 exceptional unimodal singularities corresponds to the natural partial ordering between the Gabrielov numbers, see \cite{BMM}.
\end{remark}

A comparison of Fig.~\ref{FigPyramidR} restricted to the 14 exceptional unimodal singularities, Fig.~\ref{FigPyramidAd}, and Arnold's strange duality yields the following corollary of Theorem~\ref{thm:main}.

\begin{corollary}
Denote Arnold's strange duality by $X \leftrightarrow X^\ast$.  If a singularity $Y$ is a reduction of $X$, then the dual singularity $X^\ast$ is s-adjacent to the dual singularity $Y^\ast$.
\end{corollary}

\section*{Conflict of interest statement}
On behalf of all authors, the corresponding author states that there is no conflict of interest. 


\end{document}